\documentclass[a4paper, 12pt]{article}
\usepackage{graphicx} 
\usepackage{bussproofs}
\usepackage[utf8]{inputenc}
\usepackage[english]{babel}
\usepackage{authblk}
\usepackage{amsmath}

\usepackage{ragged2e}

\usepackage{hyperref}
\usepackage{multicol}

\usepackage{xcolor}
\usepackage{hyperref}

\usepackage{setspace}
\usepackage{amssymb} 
\usepackage{amsthm} 
\usepackage{tikz}
\usepackage{adjustbox}
\usepackage{multicol}
\usepackage{proof}
\usepackage[numbers]{natbib}

\usepackage{lineno}

\theoremstyle{definition}
\newtheorem{defn}{Definition}[section]
\newtheorem{teom}[defn]{Theorem}

\newtheorem{prop}[defn]{Proposition}
\newtheorem{rem}[defn]{Remark}

\newtheorem{exem}[defn]{Example}

\usepackage[shortlabels]{enumitem}

\label{xEXPRESSIONS}

\label{xLOGICS}{}

\newcommand{\lfis}{\textit{LFI}s}
\newcommand{\lfi}{\textit{LFI}}

\newcommand{\letfp}{$LET^{+}_{F}$}
 
\newcommand{\letfm}{\ensuremath{{LET^{-}_{F}}}} 
\newcommand{\letkp}{$LET^{+}_{K}$}
\newcommand{\letf}{$LET_{F}$}

\newcommand{\letj}{$LET_{J}$}
\newcommand{\fde}{$FDE$}
\newcommand{\nel}{\textit{N4}}
\newcommand{\lets}{\textit{LET}s}

\newcommand{\lfiu}{\textit{LFIU}}

\newcommand{\lfus}{\textit{LFU}s}

\newcommand{\lz}{\ensuremath{\mathcal{L}_0}}

\newcommand{\nfm}{\ensuremath{N_{f^-}}}

\newcommand{\gfm}{\ensuremath{G\mathit{1}_{f^-}}}

\newcommand{\gtfm}{\ensuremath{G\mathit{3}_{f^-}}}

\newcommand{\gletj}{\ensuremath{GLET_J}}

\newcommand{\hf}{$\mathcal{H}$-formula}

\label{xCOLORS} %

\label{xSYMBOLS}

\newcommand{\con}{\ensuremath{{\circ}}}

\newcommand{\defi}{\stackrel{\text{\tiny def}}{=} }

\label{xGENERAL COMMANDS}%

\newcommand{\mh}{\noindent}

\newcommand{\m}{\vspace{1mm}}
\newcommand{\mm}{\vspace{2mm}}
\newcommand{\mmm}{\vspace{3mm}}

\newcommand{\setl}{\setlength\itemsep{-0.2em}}

\newcommand{\vv}{\vspace{-1mm}}

\newcommand{\bqu}{\begin{quote} \small}
\newcommand{\equ}{\end{quote} \normalsize}
\newcommand{\enr}{\begin{enumerate}[label={(\arabic*)}, resume]}
\newcommand{\eenr}{\end{enumerate}}

\usepackage{graphicx}
\makeatletter
\providecommand*{\ndashv}{%
  \mathrel{%
    \mathpalette\@ndashv\nvdash
  }%
}
\newcommand*{\@ndashv}[2]{%
  \reflectbox{$\m@th#1#2$}%
}
\makeatother

\title{\bf A Kleene-inspired sequent system for a minimal logic of evidence and truth}  

\author{N. Grandinetti$^1$,  A. Rodrigues$^2$ \\ 
  Federal University of Minas Gerais  \\
 $^1$ ninogl@ufmg.br \\
$^2$   abilio@ufmg.br \\ 
} 

\begin{document} 

\maketitle

\mm 
\begin{center}
\emph{This paper is a small tribute to Professor Luiz Carlos Pereira,  a great logician and  Flamengo supporter -- a sure sign of wisdom and good taste -- in recognition of his intellectual influence and a long-standing friendship with one of the authors, dating back to 2002.}
\end{center}

\m\m 

\begin{abstract}
\mh This paper investigates the proof theory of the logic of evidence and truth \letfm, with the aim of obtaining a sequent calculus consisting only of logical rules, suitable for a simple bottom-up decision procedure. We begin with \gfm, a sequent system containing structural rules, which makes explicit the main difficulties in obtaining a fully logical and decidable calculus for \letfm. We then introduce \gtfm, a \textit{G3}-style system obtained from \gfm\ by applying Kleene’s trick, which absorbs contraction into the logical rules by adding the principal formula of the rule to its premises. 
\end{abstract}


\onehalfspacing

\section{Introduction}

Logics of evidence and truth (\lets) form a family of paraconsistent and paracomplete logics. 
They extend  the logic of first-degree entailment (\fde), also known as Belnap–Dunn four-valued logic  (e.g. \cite{belnap.1977.how,dunn76}), with a unary operator \con\  designed to recover classical negation  to formulas within its scope. 
The sentential logic \letfm\ is a minimal \textit{LET}, obtained by adding \con\ and its inference rules to \fde.


This paper investigates the proof theory of \letfm. More precisely, our aim is to present a sequent calculus for \letfm\ with purely logical rules, and to use it to obtain a simple bottom-up decision procedure.
We do so by first introducing the system \gfm, a  sequent system which contains the  structural rules and highlights the problems that must be addressed in order to get to our aimed fully logical and decidable sequent system for \letfm. 
The system \gfm\ is then modified into a system that we call \gtfm, a \textit{G3}-style system inspired by Kleene \cite[\S80, pp.~479ff.]{Kleene1952}. \gtfm\ is obtained from \gfm\ by applying the so-called Kleene's trick, which consists in duplicating the principal formula in the premises so that contraction is absorbed into the logical rules.

The remainder of this text is structured as follows.
Section~\ref{sec.on.lets} provides a brief presentation of   \textit{LET}s, their motivations, and their historical background. Readers interested primarily in proof theory may skip this section and proceed directly to 
 Section~\ref{sec.gfm}, which   
 begins by presenting a natural deduction system for \letfm, called \nfm,  together with a sound and complete bivalued semantics. \nfm\  is the sentential fragment of a first-order version of \letfm, investigated in \cite{Rodrigues2022}.
 In Section~\ref{sec.gfm}   we also  introduce the sequent system \gfm, and show, based on results from \cite{Coniglio2025}, that the cut rule can be eliminated from deductions in \gfm.
In Section~\ref{sec.gtfm} we introduce the cut-free system \gtfm, and show that all structural rules are admissible in it. 
In this Section we also   establish the main result of this paper, namely,  the decidability of \letfm\ by means of a bottom-up decision method for \gtfm. 
Finally, in Section~\ref{sec.final.rem}, we   outline directions for future research.

\section{On the logics of evidence and truth}  \label{sec.on.lets}

We begin by formally characterizing logics of formal inconsistency (\lfis) and logics of formal undeterminedness 
(\lfus). Logics of evidence and truth are simultaneously \lfis\ and \lfus. 
   
 Let \textbf{L} be a Tarskian, compact, and structural logic 
 equipped with a negation $\neg$, a  standard disjunction $\lor$, and a unary connective    $\con$. 
Consider the following conditions: 
\begin{itemize}\setl  
       \item [(i)] $ A ,\neg A \nvdash B $ for some $ A $ and $ B $;
 \item [(ii)]    
       $\circ A , A  \nvdash  B $  and $\circ A , \neg  A  \nvdash  B $ for some $ A $ and $ B $; 
       \item [(iii)]  $\circ A , A ,\neg A \vdash B $ for every $ A $ and $ B $. \m 
       \item [(iv)] $\nvdash A  \lor \neg A $ for some $ A $;
       \item [(v)] 
       $\con A  \nvdash  A $ and 
       $\con A  \nvdash \neg  A $ for some $ A $; 
        \item [(vi)]  $\con A  \vdash  A  \lor \neg A $ for every $ A $. 
\end{itemize}

\mh 
 ${\bf L}$ is a {\it logic of formal inconsistency} (\lfi) if conditions (i)–(iii) hold, and a {\it logic of formal undeterminedness} (\textit{LFU}) if conditions (iv)–(vi) hold. Conditions~(ii) and (v) are needed to ensure that (iii) and (vi) hold in a non-trivial way.
    If all the above conditions hold, \textbf{L} is a
    \textit{logic of formal inconsistency and undeterminedness}, i.e. an \textit{LFIU}.  

The terminology \textit{LFI} and \textit{LFU} was introduced by Carnielli \cite{car.mar.2002.tax} and Marcos \cite{marcos2005}. 
The underlying ideas, however, can be traced back to da Costa's earlier work: 
\lfis\ further develop the $C_n$ hierarchy of paraconsistent logics introduced in the 1960s \cite{costa1963}, 
and the recovery of the law of excluded middle in a way analogous to explosion was already investigated in the 1980s 
\cite{daCosta1989,daCostaMarconi1986}.
\textit{LFI}s and \textit{LFU}s have been investigated by several authors, especially by those associated with the so-called Brazilian school of paraconsistency (see e.g.  \cite{car.con.2016.book,car.mar.2002.tax,car.con.mar.2007,marcos.tese,MortariRodriguesSeccoPimentelDottaviano2025}).
The idea of a logic of evidence and truth first appeared in \cite[Section~3.3.2]{kolkata}, where a minimal \textit{LFI} called $mbC$ was modified so as to be faithful to the intended interpretation in terms of conclusive and non-conclusive evidence. 
\lets, as well as \lfis\ and \lfus, thus, belong to the evolutionary line of paraconsistent and paracomplete logics initiated by da Costa’s pioneering work, which introduced systems with two notions of consequence: one sub-classical and one classical, the latter being recovered by means of a \con-style operator.

Let us now provide a more precise characterization of logics of evidence and truth.
They are \lfiu s that have \fde\ as a fragment, and they come with an intended intuitive interpretation in terms of evidence that may fail to be conclusive (see e.g. \cite{letj,Rodrigues2022,rod.car.llp}). In this setting, $\con A$ expresses that whatever evidence is available for $A$, whether positive or negative, is taken to be conclusive; when $\con A$ does not hold, the evidence is non-conclusive.
The intuitive interpretation assumes that conclusive evidence is governed by classical logic. 
Thus, a formula $\con A$ implies that classical reasoning holds for $A$, 
as well as for all formulas built from $A$ using the connectives $\land$, $\lor$, and $\neg$. 
In \citep{letf}, an alternative but closely related intuitive interpretation in terms of information 
has been proposed for \lets.   Accordingly, 
 a sentence $ \con A $ means that the information conveyed by $A$, positive or negative, is reliable, and 
 when $\con A$ does not hold, it means that there is no reliable information about $A$. 
This reading   appears as a natural development of the information-based interpretation of \fde, introduced by Belnap and Dunn in the 1970s (see e.g. \cite{belnap.1977.how,dunn76}).

The logic \letfm\ is a  minimal \textit{LET}, obtained by 
simply adding to \fde\   rules validating the inferences below:
\vv \begin{itemize}
 \setl 
\item[] (1) $ \con A,A,\neg A \vdash  B$ and  
\item[] (2) $ \con A \vdash A\lor\neg A$,  
\end{itemize} \vv 
\mh which recover classical  logic  for sentences in the scope of \con. 
\letfm\ is the $\to$-free fragment of the logic \letj,   introduced in \cite{letj} 
and further investigated in \cite{axioms,Rodrigues2022}. 
\letj\ is equipped with a constructive implication -- it is obtained by adding to Nelson's logic \nel\ \cite{almu.nel.84} 
the inferences (1) and (2) above. The proof theory of \letj\ was investigated in \cite{Rodrigues2022}, which provides the basis for some of the results established here. 
\letfm\ is also  the $\bullet$-free 
fragment of the logic \letf, investigated in \cite{letf} ($\bullet$ is a non-classicality operator: 
it has a deductive behavior dual to \con\ and $\bullet A$ means that there is no conclusive evidence about $A$ -- see \cite[Section~3.1]{letf}).
A first-order version of \letfm\ has been studied in \cite{Rodrigues2022}.

The notion of evidence for a sentence $A$ that underlies the intuitive interpretation of \lets\  can ben explained  as `reasons for believing in or accepting $A$'. These reasons may be non-conclusive or wrong,
 and so   evidence for $A$ does not imply the truth of $A$. There may be
simultaneous conflicting evidence for sentences  $A$ and $\neg A$,
  and  evidence for $A$ is meant to be objective in the sense that it is independent from
the belief of an agent in $A$. 
It must be observed that accepting a contradiction in a context of reasoning is not to be understood as a rational acceptance: the contradictory evidence is merely admitted, in the sense that one should be able to reason in its presence.
The interpretation in terms of information takes the definition of information as \textit{meaningful data} (see \cite{dunn2008,fetzer2004a}), which allows information to be false.  
A more detailed discussion of the notion of evidence and information can be found in \cite[Section~2]{rod.car.llp},  
and an exposition  that emphasizes the interpretation in terms of information can be found in \cite[Sections~2 and 3]{Rodrigues2022}  and \cite[Section~2]{qletf}.

\section{The logic \letfm\ and the   system \gfm} \label{sec.gfm}

Consider a denumerable set ${\mathcal{V}}$  of sentential letters and the set  of connectives  
$\mathcal{S}_0 = \{ \land, \lor, \neg,\con \}$. 
The language $\mathcal{L}_0$  is  the set of formulas  generated by 
${\mathcal{V}}$ over $\mathcal{S}_0$.    
Roman capitals $A, B, C,\dots$ will be used as metavariables for the formulas of $\mathcal{L}_0$, 
 while Greek capitals $\Gamma, \Delta, \Sigma, \dots$ will be used as metavariables for finite sets of formulas.


\begin{defn}  (The natural deduction system \nfm)   \label{def.nfm}

\m\mh 
The deductive system \nfm\ is defined over the language \lz\    
by the following rules: 

\m\mh 

 \begin{center}

\AxiomC{${\cal D}_1$}
\noLine
\UnaryInfC{$A$}
\AxiomC{${\cal D}_2$}
\noLine
\UnaryInfC{$B$}
\RightLabel{$\wedge I$}
\BinaryInfC{$A \wedge B$}
\DisplayProof \hspace{1cm}
\AxiomC{${\cal D}$}
\noLine
\UnaryInfC{$A\wedge B$}
\RightLabel{\rm $\wedge$E}
\UnaryInfC{$A$}
\DisplayProof 
\AxiomC{${\cal D}$}
\noLine
\UnaryInfC{$A\wedge B$}
\UnaryInfC{$B$}
\DisplayProof 

\mm 

\AxiomC{${\cal D}$}
\noLine
\UnaryInfC{$A$}
\RightLabel{\rm $\vee$I}
\UnaryInfC{$A\vee B$}
\DisplayProof
\AxiomC{${\cal D}$}
\noLine
\UnaryInfC{$B$}
\UnaryInfC{$A\vee B$}
\DisplayProof \hspace{1cm}
\AxiomC{}
\noLine
\UnaryInfC{${\cal D}_1$}
\noLine
\UnaryInfC{$A\vee B$}
\AxiomC{\rm[$A$]$^u$}
\noLine
\UnaryInfC{${\cal D}_2$}
\noLine
\UnaryInfC{$C$}
\AxiomC{\rm [$B$]$^v$}
\noLine
\UnaryInfC{${\cal D}_3$}
\noLine
\UnaryInfC{$C$}
\RightLabel{\rm $\vee$E,$u$,$v$}
\TrinaryInfC{$C$}
\DisplayProof 

\mm 

\AxiomC{${\cal D}$}
\noLine
\UnaryInfC{$\neg A$}
\RightLabel{\rm $\neg\wedge$I}
\UnaryInfC{$\neg(A\wedge B)$}
\DisplayProof
\AxiomC{${\cal D}$}
\noLine
\UnaryInfC{$\neg B$}
\UnaryInfC{$\neg(A\wedge B)$}
\DisplayProof \hspace{1cm}
\AxiomC{}
\noLine
\UnaryInfC{${\cal D}_1$}
\noLine
\UnaryInfC{$\neg(A\wedge B)$}
\AxiomC{\rm[$\neg A$]$^u$}
\noLine
\UnaryInfC{${\cal D}_2$}
\noLine
\UnaryInfC{$C$}
\AxiomC{\rm [$\neg B$]$^v$}
\noLine
\UnaryInfC{${\cal D}_3$}
\noLine
\UnaryInfC{$C$}
\RightLabel{\rm $\neg\wedge$E,$u$,$v$}
\TrinaryInfC{$C$}
\DisplayProof 

\mm 

\AxiomC{${\cal D}_1$}
\noLine
\UnaryInfC{$\neg A$}
\AxiomC{${\cal D}_2$}
\noLine
\UnaryInfC{$\neg B$}
\RightLabel{\rm $\neg\vee$I}
\BinaryInfC{$\neg(A \vee B)$}
\DisplayProof \hspace{1cm}
\AxiomC{${\cal D}$}
\noLine
\UnaryInfC{$\neg(A\vee B)$}
\RightLabel{\rm $\neg\vee$E}
\UnaryInfC{$\neg A$}
\DisplayProof 
\AxiomC{${\cal D}$}
\noLine
\UnaryInfC{$\neg(A\vee B)$}
\UnaryInfC{$\neg B$}
\DisplayProof

\mm

\AxiomC{${\cal D}$}
\noLine
\UnaryInfC{$A$}
\RightLabel{\rm $\neg\neg$I}
\UnaryInfC{$\neg \neg A$}
\DisplayProof \hspace{1cm}
\AxiomC{${\cal D}$}
\noLine
\UnaryInfC{$\neg\neg A$}
\RightLabel{\rm $\neg\neg$E}
\UnaryInfC{$A$}
\DisplayProof

 \mmm 

\AxiomC{${\cal D}_1$}
\noLine
\UnaryInfC{$\circ A$}
\AxiomC{${\cal D}_3$}
\noLine
\UnaryInfC{$\neg A$}
\AxiomC{${\cal D}_2$}
\noLine
\UnaryInfC{$A$}
\RightLabel{\rm $EXP^{\circ}$}
\TrinaryInfC{B}
\DisplayProof \hspace{2cm}
\AxiomC{}
\noLine
\UnaryInfC{${\cal D}_1$}
\noLine
\UnaryInfC{$\circ A$}
\AxiomC{\rm[$A$]$^u$}
\noLine
\UnaryInfC{${\cal D}_2$}
\noLine
\UnaryInfC{$B$}
\AxiomC{\rm [$\neg A$]$^v$}
\noLine
\UnaryInfC{${\cal D}_3$}
\noLine
\UnaryInfC{$B$}
\RightLabel{{$PEM^{\circ}$},$u$,$v$}
\TrinaryInfC{$B$}
\DisplayProof

 \end{center}

\m\mh 
The assumption classes [$A$]$^u$,  [$B$]$^v$, [$\neg A$]$^u$, [$\neg B$]$^v$ contain open assumptions of the deductions of the premises of the final inference, but are closed in the whole deduction. Also, we define major and minor premises as usual, except in the rule $EXP^{\circ}$, where there is no notion of major or minor premises.

The definition of a derivation of $A$ from $\Gamma$ in $\nfm$ is the usual one for natural
deduction systems (recall that we assume here that $\Gamma$ is a finite set).
It suffices to say that a derivation $\mathcal D$ is a tree of formulas
whose nodes are either elements of $\Gamma$ or are obtained from earlier nodes by applications of the rules above, and whose bottom node is its conclusion.
We write $\Gamma \vdash_{\nfm} A$ to mean that there exists a derivation in $\nfm$ from the premises
in $\Gamma$ with conclusion $A$, omitting the subscript when there is no risk of confusion.

 \end{defn}

\mm

\begin{defn} ({Valuation semantics for \letfm})  \label{def.val.smt.letfm}

\m 
\mh
A bivalued semantics for {\letfm} is given by the class of functions
$v:\lz \to \{0,1\}$ satisfying the following clauses:

\vv\vv\vv     \begin{itemize} \setl 
\item[v1.] $v( A  \land  B )=1$ iff $v( A )=1$ and $v( B )=1$,
\item[v2.] $v( A  \lor  B )=1$ iff $v( A )=1$ or $v( B )=1$,
\item[v3.] $v(\neg ( A  \land  B ))$ = 1 iff $v(\neg  A ) = 1$ or  $v(\neg  B ) = 1$,
\item[v4.] $v(\neg ( A  \lor  B )) = 1$ iff $v(\neg  A ) = 1$ and  $v(\neg  B ) = 1$,
\item[v5.] $v( A ) = 1$ iff $v(\neg\neg  A ) = 1$, 
\item[v6.] If $v(\con A) = 1$, then  $v(A) = 1$ if and only if $v(\neg A) = 0$.  
\end{itemize} 
 
 \vv\mh We say that a formula $A$ is a semantic consequence of $\Gamma$, written
$\Gamma \vDash_{\letfm} A$, if and only if for every valuation $v$, whenever
$v(B)=1$ for all $B \in \Gamma$, then $v(A)=1$. 
\end{defn}

As remarked in the Introduction, the logic \letfm\ is a fragment of both the sentential logics \letj\ and  \letf, 
introduced in \cite{letj} and \cite{letf} respectively. \letf\  extends \letfm\ with the operator $\bullet$, 
already mentioned in Section~\ref{sec.on.lets}, 
together with its corresponding rules. 
The semantics presented above is obtained simply by suppressing the clause for the $\bullet$ connective in the semantics of \letf\ \cite[Definition~26]{letf}. 
 A natural deduction system for \letfm\ is obtained just by dropping the $\bullet$-rules from the deductive system of 
 \letf\  \cite[Definition~23]{letf}. 
The system presented here differs from the $\bullet$-free system of \cite{letf} 
only with respect to the rule that recovers excluded middle, but both are equivalent. 
A system  identical to \nfm\  is the sentential fragment of the first-order version of  \letfm, 
investigated in \cite{Rodrigues2022}.  

\mm  \nfm\ is sound and complete with respect to the bivalued  semantics presented in Definition~\ref{def.val.smt.letfm}.

\begin{teom}  \label{th.sound.comp.nfm} \ 

\m\mh 
$\Gamma \vdash_{\nfm} A$ if and only if $\Gamma \vDash_{\letfm} A$

\begin{proof}
    Since the logic \letfm\ is contained in \letf, the soundness and completeness of the system \nfm\ with respect to the semantics above follow immediately from the corresponding soundness and completeness results for \letf\ \cite[Theorem~28]{letf}.
\end{proof}    
\end{teom}

\subsection{The sequent calculus \gfm}  \label{sec.subsection.gfm}

Below, we present the sequent system \gfm,  in which all structural rules are explicit. 
We discuss the features of this system and show its limitations with respect to our purposes. Later, 
in Section \ref{sec.gtfm}, we will introduce the system \gtfm, with no explicit structural rules.\footnote{In calling these systems \textit{G1} and \textit{G3}, we follow the terminology
introduced by Kleene \cite{Kleene1952} and later adopted in standard textbooks
on proof theory, such as those by Troelstra and Schwichtenberg
\cite{Troelstra2000} and by Negri and von Plato \cite{Negri2001}.
Here, we use the numeral following \textit{G} to indicate whether the system can dispense with \textit{all} structural rules: “1” indicates that it cannot, whereas “3” indicates that it can.}


 

\begin{defn} (The system \gfm)  \label{def.gfm}

\m\mh   
The system \gfm\ is defined over the language \lz. 
Greek capitals $\Gamma, \Delta, \Sigma$  now denote multisets of formulas of formulas of \lz, and \(A, B\) range over formulas of \lz.
The notions of principal, active and context formulas are understood in the standard way.

\mm\mh 
\textit{Axioms:}
\begin{center}
$\infer[Id] {A\Rightarrow A} {}$
\end{center}

\mh \textit{Structural rules:}
\begin{center}
$\infer[LW] {A, \Gamma \Rightarrow C} {\Gamma \Rightarrow C}$
\quad 
$\infer[Cut] {\Gamma^{'}, \Gamma \Rightarrow C} { \Gamma^{'} \Rightarrow A & A, \Gamma \Rightarrow C}$
\quad
$\infer[LC] {A, \Gamma \Rightarrow C} {A,A,\Gamma \Rightarrow C}$
\end{center}

\mh \textit{Logical rules:}
\begin{center}
$\infer[L\land] {A \land B,\Gamma \Rightarrow C} {A,B , \Gamma \Rightarrow C}$ 
\quad 
$\infer[R\land] {\Gamma \Rightarrow A \land B} { \Gamma\Rightarrow A   &   \Gamma\Rightarrow B}$

\

$\infer[L\lor] {A \lor B,\Gamma \Rightarrow C} {A,\Gamma \Rightarrow C & B,\Gamma \Rightarrow C}$
\quad 
$\infer[R\lor] {\Gamma \Rightarrow A \lor B } { \Gamma\Rightarrow A }$  $ \infer[] {\Gamma \Rightarrow A \lor B } { \Gamma\Rightarrow B}$

\

$\infer[L\neg\lor] {\neg(A \lor B),\Gamma \Rightarrow C} {\neg A, \neg B,\Gamma \Rightarrow C}$
\quad 
$\infer[R\neg\lor] {\Gamma \Rightarrow \neg(A \lor B)} { \Gamma\Rightarrow \neg A   &   \Gamma\Rightarrow \neg B}$

\

$\infer[L\neg\land] {\neg(A \land B),\Gamma \Rightarrow C} 
{\neg A,\Gamma \Rightarrow C & \neg B,\Gamma \Rightarrow C}$
\quad 
$\infer[R\neg\land] {\Gamma \Rightarrow \neg(A \land B)} {\Gamma\Rightarrow \neg A}$ 
$\infer[] {\Gamma \Rightarrow \neg(A \land B)} { \Gamma\Rightarrow \neg B}$

\

$\infer[{L\circ}_1] {\circ A, \Gamma \Rightarrow C}{A, \Gamma\Rightarrow C & \neg A, \Gamma\Rightarrow C }$
\quad
$\infer[{L\circ}_2] {\circ A, \Gamma \Rightarrow C}{ \Gamma\Rightarrow A & \Gamma\Rightarrow \neg A }$

\

$\infer[L\neg\neg] {\neg\neg A , \Gamma \Rightarrow C}{A , \Gamma \Rightarrow C}$
\quad 
$\infer[R\neg\neg] {\Gamma \Rightarrow \neg\neg A} {\Gamma \Rightarrow A}$

\end{center}

\m\mh 
A \emph{derivation} in \gfm\  is defined in the usual 
way as a finite tree of sequents such that each node of the tree is a sequent, the leaves
are initial sequents of the calculus (axioms), and each non-leaf node is obtained from its immediate 
predecessors by an application of one of the inference rules. The bottom
sequent of the tree is called the conclusion of the derivation. A sequent
$\Gamma\Rightarrow C$ is said to be derivable in the \gfm\  if there exists a derivation
whose conclusion is $\Gamma\Rightarrow C$.

\end{defn}

We now introduce some definitions that will be useful in the remainder of the paper.

 \begin{defn} (Generalized subformulas, formula weight and proof height)

\m\mh A \emph{literal}, $l$, is either an atomic formula $p$ or the negation $\neg p$ of an atomic formula~$p$.

\begin{enumerate}

\item The \emph{weight} of a formula $A$, denoted $w(A)$, is defined as follows:

\begin{enumerate}%
		\item  $ w(l) = 0$, where \textit{l} is a literal;
		
            \item $w(\neg \neg A) =  w(A)  + 1$;
		\item  $w(A \ast B) =   w(A) +  w(B)  + 1$, for $\ast \in \{ \land, \lor \}$;
		\item $w(\neg(A \ast B)) =   w(\neg A) +  w(\neg B)  + 1$,  
	 for $\ast \in \{\land,\lor \}$;
\item $w(\circ A) =   w(A) +  w(\neg A)  + 1$;

		\end{enumerate}

\item We define the notion of \emph{generalized subformula} as follows:

\begin{enumerate} %
 \item $A$ is 
a generalized subformula of $A$;
\item $A$ is 
a generalized subformula of $\neg A$;
\item $A$ and $B$ are    generalized subformulas of both  $A \ast B$,     
$\ast \in \{ \lor, \land \}$;
\item $\neg A$ and $\neg B$ are    generalized subformulas of  $\neg (A \ast B)$, 
$\ast \in \{ \lor, \land \}$;  
\item $A$ and $\neg A$ are    generalized subformulas of $\circ A $. 

\end{enumerate}

\item 
The \emph{height} of a derivation $\mathcal{D}$, denoted $h(\mathcal{D})$, is defined as follows:
\begin{enumerate}
  \item If $\mathcal D$ is a derivation consisting of an initial (axiomatic) sequent, then
  $h(\mathcal D)=0$.
  \item If $\mathcal D$ is a derivation of a sequent $S$ obtained from a single premise
  $S'$ by an inference rule, and $\mathcal D'$ is a derivation of $S'$, then
  $h(\mathcal D)=h(\mathcal D')+1$.
  \item If $\mathcal D$ is a derivation of a sequent $S$ obtained from two premises
  $S'$ and $S''$ by an inference rule, and $\mathcal D'$ and $\mathcal D''$ are derivations
  of $S'$ and $S''$, respectively, then
  $h(\mathcal D)=\max(h(\mathcal{D}'), h(\mathcal{D}'')) +1$.
\end{enumerate}

\noindent The definition extends in the obvious way to rules with more than two premises by taking the maximum over all immediate subderivations. 
Also, we can easily adapt this definition to natural deduction settings by replacing the word ``sequent'' 
  by the word ``formula'', and, in item (a), replacing ``initial (axiomatic) sequent'' for ``open assumption''.

 \end{enumerate}   
 \end{defn}

\begin{teom} \label{th.cut.elim.gfm} \ 

\m\mh 
    For every derivation in \gfm\ that uses the rule \emph{Cut}, there exists a derivation of the same end-sequent that does not use the rule \emph{Cut}.
\begin{proof}
In \cite{Coniglio2025}, cut-elimination is established for ${GLET}_J$ through a proof of elimination of multicut. The logical rules of \gfm\ are precisely the $\to$-free fragment of those of \gletj.  
While \gletj\ treats contexts as sets, thus making explicit contraction 
   unnecessary, \gfm\   uses multisets and therefore requires an explicit contraction rule.  
Adapting the cut-elimination proof from \gletj\ (with sets) to \gfm\ (with multisets) is essentially routine, and the use of multicut accounts for the known problems caused by contraction.
      \end{proof}
\end{teom}


The system \gfm\ is inspired by the systems \textit{G1i} and \textit{G3i} presented in \citet[Chapter~3]{Troelstra2000}. It extends \textit{G1i} with rules for $\neg\land$ and $\neg\lor$, which behave as the positive rules for $\lor$ and $\land$ (with $\land$ formulated as in \textit{G3i}), and with rules for \con, taken from $GLET_J$, a sequent calculus for the logic of evidence and truth \letj, introduced in \cite{Coniglio2025}.


{By contrast with \textit{G1i} and \textit{G3i}, which are able to prove the elimination of the non-atomic axiom rule (see \cite[pp. 65]{Troelstra2000}), in \gfm\ we cannot restrict the axiom rule to literals as principal formulas.   The reason is that, under such a restriction, we cannot derive $\con A \Rightarrow \con A$, since there is no right rule for $\con$.}

{Besides highlighting the peculiarities of a sequent system for \letfm, there is another reason for introducing \gfm\ first. Its logical rules provide a more convenient setting for proving its equivalence with \nfm. Indeed, this equivalence is established in Section~\ref{sec.equiv.gfm.nfm} below, and soundness and completeness with respect to the bivaluation semantics of Definition~\ref{def.val.smt.letfm} follow from it.}

  
In the following we show some constructive features of \gfm\    
and we prove the equivalence between the natural deduction system \nfm\ and \gfm. 


\begin{prop}  (Disjunction property and constructive falsity)  \label{prop.dis,pro}

\m\mh (i) Disjunction property: if $\Rightarrow A \lor B$ is derivable in \gfm,
then either $\Rightarrow~A$~or~$\Rightarrow B$ is derivable in \gfm.

\m\mh (ii) Constructive falsity: if $\Rightarrow \neg(A \land B)$ is
derivable in \gfm, then either $\Rightarrow \neg A$ or $\Rightarrow \neg B$
is derivable in \gfm.

\begin{proof}
For (i), given a cut-free derivation of $\Rightarrow A \lor B$, the only rule that can introduce the formula $A \lor B$ on the right is $R\lor$. 
For (ii), given a cut-free derivation of $\Rightarrow \neg(A \land B)$, the only rule that can introduce the formula $\neg(A \land B)$ on the right is $R\neg\land$.
 \end{proof}   
\end{prop}

\m 

We can strengthen this result by adapting the notion of Harrop formulas to our setting.

\begin{defn}
The class $\mathcal{H}$ of Harrop-like formulas is defined as follows:
\begin{enumerate} \setl 
    \item Every literal is an $\mathcal{H}$-formula.
    \item   If $A$ and $B$ are $\mathcal{H}$-formulas, then $A \land B$ is a $\mathcal{H}$-formula. 
    \item If $\neg A$ and $\neg B$ are $\mathcal{H}$-formulas, then $\neg(A \lor B)$ is a $\mathcal{H}$-formula.  
    \item  For any $A$, $\neg \con A$ is an $\mathcal{H}$-formula. 
    \item Nothing else is an $\mathcal{H}$-formula.
\end{enumerate}
\end{defn}

\begin{prop} \label{prop.str.disj} \ 

 \m\mh (i) Strong disjunction property: if $\Gamma\Rightarrow A \lor B$ is derivable in \gfm\ and $\Gamma$ consists only of $\mathcal{H}$-formula,
then either $\Gamma\Rightarrow A$ or $\Gamma\Rightarrow B$ is derivable in~\gfm.

\m\mh (ii) Strong constructive falsity: if $\Gamma\Rightarrow \neg(A \land B)$ is
derivable in \gfm\ and $\Gamma$ consists only of $\mathcal{H}$-formula, then either $\Gamma\Rightarrow \neg A$ or $\Gamma\Rightarrow \neg B$
is derivable in~\gfm.

\begin{proof} \ 

\m\mh (i) If $n = 0$, then the derivation is an axiom, and hence it must have the form
$A \Rightarrow A$. Thus, the end-sequent cannot be
$\Gamma \Rightarrow A \lor B$ with $\Gamma$ consisting only of
$\mathcal{H}$-formulas.
 Now suppose   $n>0$ and $\Gamma$ is composed of $\mathcal{H}$-formulas. 
If the last rule applied is  $R\lor$, then
$\Gamma \Rightarrow A$ or $\Gamma \Rightarrow B$, respectively, is a premise of
the last inference.
 If the last rule is  $L\neg\lor$, then $\Gamma = {\neg(C \lor D)},~{\Gamma'}$   and the premise is
$\neg C, \neg D,~\Gamma'~\Rightarrow~A \lor B$. Since $\neg(C \lor D)$ is a $\mathcal{H}$-formula,
$\neg C$ and $\neg D$ are $\mathcal{H}$-formulas as well. 
By the inductive hypothesis, either $\neg C, \neg D, \Gamma' \Rightarrow A  $ or $\neg C,\neg  D, \Gamma' \Rightarrow   B$. 
If $LW$ was used to derive $\Gamma \Rightarrow A \lor B$, introducing an $\mathcal{H}$-formula on the left, we can proceed in the obvious way, since the induction hypothesis will apply to the premise.
 Analogous reasoning applies when the last rule is $L\land$.  
 If any other rule is the last rule, then $\Gamma$ is not composed of $\mathcal{H}$-formulas. 
 
 \m\mh (ii) This proof can be directly adapted to the case of the strong constructive falsity property 
 by simply considering the rule $R\neg\land$ in the place of $R\lor$. The cases of $L\land$ and $L\neg\lor$ are analogous.  \end{proof}

\end{prop}

\begin{rem} (On the   constructive character of \letfm)  \label{rem.harrop.etc}
    
\m\mh (i) 
The idea of a constructive negation was proposed by Nelson \cite{nelson.49}, where the first-order 
logic \textit{N} is introduced. 
 The sentential fragment of \textit{N} is the logic \textit{N3}, and \nel\ is obtained 
 by dropping explosion from \textit{N3}. 
 Nelson’s motivation was to provide a logic that could serve as the underlying logic of a number theory  with constructive features lacking in intuitionistic logic (see \citet{nelson.49}, pp.~16ff.).
Both \textit{N3} and \nel\  have a constructive negation and also a constructive implication.  
That is, instead of proving the falsity of a formula $A$ by deriving a contradiction from the assumption 
$A$, one proves $\neg A$ 
directly (the notion of a proof of $\neg p$ is primitive, just like that of a proof of $p$). 
Nelson thus provided a  negation that fits the underlying idea of information-based logics, namely, that positive and negative information must be treated independently of each other (see \cite{wans.93}, pp.~13ff.). 
From the    point of view of both the evidence interpretation and the information-based interpretation, the validity of the constructive falsity property is indeed as desirable as the disjunction property.  

\m\mh (ii) 
Roughly, the underlying idea behind Harrop formulas is to preserve the constructive content of proofs. Thus, positive occurrences of disjunction are restricted, ensuring that any proof of a disjunction provides a constructive choice, i.e.\ an effective method for determining which alternative holds.
 The same idea applies to $\neg(A\land B)$, which is equivalent to $\neg A \lor\neg B$, in a 
context with constructive negation. 
  A formula $\con A$ cannot be an \hf, since it implies $A \lor \neg A$, which is equivalent to $\neg(A \land \neg A)$. Whenever either formula holds, even if restricted to one or a few instances, we can derive formulas of the form $A \lor B$, as well as formulas of the form $\neg(A \land B)$, without a constructive choice. This is incompatible with both the disjunction property and 
the constructive falsity property.
On the other hand, formulas of the form $\neg\con A$ are innocuous and do not compromise either of these properties.
It is also worth noting that  Propositions~\ref{prop.dis,pro} and \ref{prop.str.disj} above are interesting, but not at all unexpected. The logic \textit{FDE}, which is the \con-free fragment of \letfm, and the $\to$-free fragment of \nel, is constructive. On the other hand, formulas in the scope of \con\ exhibit classical, non-constructive behavior. 

\m\mh (iii) 
Note that the results of Propositions~\ref{prop.dis,pro} and~\ref{prop.str.disj} are features of \letfm, so they hold for the natural deduction system \nfm\ and are inherited by \gtfm, to be introduced in Section~\ref{sec.gtfm}.

\end{rem}

Now, we turn to the equivalence between \nfm\ and \gfm.

\subsection{Equivalence between \nfm\ and \gfm}    \label{sec.equiv.gfm.nfm}

The notation $\nfm \vdash \Gamma \Rightarrow C$ means that there exists a derivation of $C$ in \nfm\
from the open assumptions in $\Gamma$.
The notation
$\gfm \vdash \Gamma \Rightarrow C$ means that there exists a derivation of the sequent
$\Gamma \Rightarrow C$ in \gfm.
Note that although $\Gamma$ is a set of formulas in natural deduction, in the translation of a derivation
into a sequent we treat $\Gamma$ as a multiset: each open assumption class $[A]^u$ (with a distinct label)
contributes one occurrence of $A$ on the left-hand side of the corresponding sequent.

The derivability relation of \nfm\  admits left contraction and  left weakening. 
More precisely, if
$\nfm \vdash A,A,\Gamma\Rightarrow C$, then $\nfm \vdash A,\Gamma\Rightarrow C$, and if
$\nfm \vdash \Gamma\Rightarrow C$, then $\nfm \vdash A,\Gamma\Rightarrow C$.
Contraction is obtained by identifying the labels of two distinct assumption classes containing
the same formula.
Left weakening is obtained as follows: 

\begin{center}
$\infer[$LW$]{A,\Gamma \Rightarrow C}{\Gamma \Rightarrow C}$
\qquad becomes \qquad
$\infer{C}{\infer{A\land C}{[A]^u&\infer*{C}{\Gamma}}}$
\end{center}
\mh 
Left weakening corresponds to the monotonicity of the derivability relation: if a formula $C$ is derivable from the assumptions $\Gamma$, then $C$ is also derivable from
$\Gamma,A,$ where the added assumption $A$ may remain unused in the derivation.
These properties are required in order to establish the equivalence between \nfm\ and \gfm,
since \gfm\ includes both weakening and contraction as explicit structural rules.

\begin{teom}  \label{th.equiv.gn} \ 

\m \mh 
$\gfm  \vdash \Gamma \Rightarrow C$ if and only if $\nfm  \vdash \Gamma \Rightarrow C$.

\begin{proof} \ 

\m \mh Assume that  the proof of $\Gamma \Rightarrow C$ in \gfm\ is cut-free.

\m\mh From left to right, the proof proceeds by induction on the height $n$ of
derivations in \gfm.
Let   $\Gamma$ be a multiset. The notation $set(\Gamma)$ denotes the underlying set
associated with $\Gamma$, that is, the set of formulas occurring in $\Gamma$,
ignoring multiplicities.
  If $n = 0$, then $\Gamma \Rightarrow C$ is an axiom. In \nfm, this corresponds to a deduction consisting of a single node  $C$.
 If $n > 0$, let $\mathcal{D}$ be a derivation of $\Gamma \Rightarrow C$ of height $n$ in \gfm.  
We assume as induction hypothesis that the result holds for all derivations in \gfm\ of height at most $n-1$; that is, for every derivation $\mathcal{D}'$ of $\Gamma' \Rightarrow C'$ in \gfm\ of height at most $n-1$, there exists a deduction $\mathcal{D}'^*$ in \nfm\ deriving $C'$ from  ${set}(\Gamma')$.
We then consider all the inference rules of \gfm.

{\small 

\begin{center}
      
${\infer[R\neg\lor]{\Gamma\Rightarrow\neg(A\lor B)}{\deduce{\Gamma\Rightarrow\neg A}{\mathcal{D}_1}&\deduce{\Gamma\Rightarrow\neg B}{\mathcal{D}_2}}}$
\qquad\text{goes to}\qquad
${\infer[I\neg\lor]{\neg(A\lor B)}{\deduce{\neg A}{\mathcal{D}_1^*} & \deduce{\neg B}{\mathcal{D}_2^*}}}$

\end{center}

 \m 

\begin{center}
$\infer[L\neg\lor]{\neg(A \lor B), \Gamma \Rightarrow C}{\deduce {\neg A, \neg B, \Gamma \Rightarrow C}{\mathcal{D}_1}}$ 
\qquad goes to \qquad   
$\deduce{C}{\infer {\mathcal{D}^*_1} {\infer{\neg A}{\neg (A \lor B)} & \infer{\neg B}{\neg (A \lor B)}}}$
\end{center}

\m 

\begin{center}
$\infer[L\con_1]{\con A, \Gamma \Rightarrow C}{\deduce{A, \Gamma\Rightarrow C}{D_1}  &  \deduce{\neg A, \Gamma\Rightarrow C}{D_2}}$
\qquad goes to \qquad   
$\infer[PEM^\con]{C}{ \con A & \deduce{C}{\deduce{\mathcal{D}_1^*}{[A]}} & \deduce{C}{\deduce{\mathcal{D}_2^*}{[\neg A]}}}$
\end{center}
}



\mm\mh The remaining cases are left to the reader.

\mm 

\m\mh From right to left, the proof proceeds by induction on the height $n$ of deductions in \nfm.
  If $n = 0$, then the deduction in \nfm consists of a single node $C$. 
 In \gfm, $\Gamma \Rightarrow C$  corresponds to an axiom.
 If $n>0$, let $\mathcal{D}$ be a deduction in \nfm\ of height $n$ deriving $C$ from a set of assumptions $\Gamma$.
As induction hypothesis, assume that the result holds for every deduction in \nfm\ of height at most $n-1$; that is,
for any deduction $\mathcal{D}'$ of height at most $n-1$ deriving $C'$ from a set of assumptions $\Gamma'$,
there exists a derivation $\mathcal{D}'^{+}$ in \gfm\ of the sequent $\Gamma'' \Rightarrow C'$,
where $\Gamma''$ is the multiset corresponding to $\Gamma'$.
We then examine the last inference rule applied in $\mathcal{D}$ and construct a corresponding derivation in \gfm.  
 The double line denotes a (possibly empty) sequence of applications of \textit{WL} or \textit{CL}.  
Moreover, the notation $(\Gamma_1,\dots, \Gamma_n)$ stands for the multiset that contains  exactly one occurrence of each distinct formula  appearing in $\Gamma_1$, \dots,  $\Gamma_n$, that is, multiplicities in the original multisets are ignored. 

\m

{\small 
\begin{center}
$
\infer[I\land]{A\land B}{\deduce{A}{\mathcal{D}_1}& \deduce{B}{\mathcal{D}_1}}$ 
\qquad goes to \qquad    
$\infer={(\Gamma_1,\Gamma_2)\Rightarrow A\land B}{
  \infer[R\land]{\Gamma_1,\Gamma_2\Rightarrow A\land B}{
    \infer={\Gamma_1,\Gamma_2\Rightarrow A}{\deduce{\Gamma_1\Rightarrow A}{D_1^{+}}}
    &
    \infer={\Gamma_1,\Gamma_2\Rightarrow B}{\deduce{\Gamma_2\Rightarrow B}{D_2^{+}}}
  }}$
\end{center}

\m

\begin{center}
$   \infer[E\neg\lor]{\neg A}{\deduce{\neg (A \lor B)}{\cal D}}$
\qquad goes to \qquad    
$\infer[{Cut}]{\Gamma \Rightarrow \neg A}{\deduce{\Gamma \Rightarrow \neg( A \lor B)}{\mathcal{D}^+} & \infer[{L\neg\lor}]{\neg( A \lor B) \Rightarrow \neg A}{\infer[LW]{\neg A, \neg  B \Rightarrow \neg A}{\neg A \Rightarrow \neg A}}}$
\end{center}

\m 
 

\begin{center}
$\infer[EXP^\con]{C}{\deduce{\con A}{\mathcal{D}_1}  & \deduce{A}{\mathcal{D}_2} & \deduce{\neg A}{\mathcal{D}_3}}$
\qquad goes to \qquad    
$\infer {\infer= {(\Gamma_1,\Gamma_2,\Gamma_3) \Rightarrow C}{\Gamma_1,\Gamma_2,\Gamma_3 \Rightarrow C}} {\deduce{\Gamma_1 \Rightarrow \circ A}{\mathcal{D}_1^+} & {\infer [Cut] {\circ A,\Gamma_2, \Gamma_3 \Rightarrow C} {{\infer={\Gamma_2, \Gamma_3 \Rightarrow  A} {\deduce{\Gamma_2 \Rightarrow   A}{\mathcal{D}_2^+}} } & {\infer= {\Gamma_2, \Gamma_3 \Rightarrow \neg A} {\deduce{\Gamma_2 \Rightarrow   \neg A}{\mathcal{D}_3^+}} }} }}$   

\end{center}
}

\m  

\mh The remaining cases are left to the reader.
\end{proof}
\end{teom}

\section{The cut-free sequent system \gtfm}  \label{sec.gtfm}


We now introduce \gtfm, a sequent system for \letfm\ consisting solely of logical rules.
Its design follows the style of the purely logical systems proposed by Kleene in \cite[\S80]{Kleene1952}. 
The concepts of context, side and principal formulas remains the same, but with the additional detail that the occurrence of the principal formula on the premise(s) of the logical rules is also called a active formula.

\begin{defn} (The system \gtfm)    \label{def.g3fm}

\m \mh 
\textit{Axioms:}
\begin{center}
$\infer[Id] {\Gamma, A\Rightarrow A} {}$
\end{center}

\mh \textit{Logical rules:}
\begin{center}
$\infer[L\land] {A \land B,\Gamma \Rightarrow C} {A,B, A \land B, \Gamma \Rightarrow C}$ 
\quad 
$\infer[R\land] {\Gamma \Rightarrow A \land B} { \Gamma\Rightarrow A   &   \Gamma\Rightarrow B}$

\

$\infer[L\lor] {A \lor B,\Gamma \Rightarrow C} {A,A \lor B, \Gamma \Rightarrow C & B,A \lor B,\Gamma \Rightarrow C}$
\quad 
$\infer[R\lor] {\Gamma \Rightarrow A \lor B } { \Gamma\Rightarrow A }$ 
$ \infer[] {\Gamma \Rightarrow A \lor B } { \Gamma\Rightarrow B}$

\

$\infer[L\neg\lor] {\neg(A \lor B),\Gamma \Rightarrow C} {\neg A, \neg B,\neg(A \lor B),\Gamma \Rightarrow C}$
\quad 
$\infer[R\neg\lor] {\Gamma \Rightarrow \neg(A \lor B)} { \Gamma\Rightarrow \neg A   &   \Gamma\Rightarrow \neg B}$

\

$\infer[L\neg\land] {\neg(A \land B),\Gamma \Rightarrow C} 
{\neg A,\neg(A \land B),\Gamma \Rightarrow C & \neg B,\neg(A \land B),\Gamma \Rightarrow C}$

 \ 
 
$\infer[R\neg\land] {\Gamma \Rightarrow \neg(A \land B)} {\Gamma\Rightarrow \neg A}$  
$\infer{\Gamma \Rightarrow \neg(A \land B)} { \Gamma\Rightarrow \neg B}$

\

$\infer[{L\circ}_1] {\circ A, \Gamma \Rightarrow C}{A, \circ A, \Gamma\Rightarrow C & \neg A,\circ A, \Gamma\Rightarrow C }$
\quad
$\infer[{L\circ}_2] {\circ A, \Gamma \Rightarrow C}{ \circ A,\Gamma\Rightarrow A & \circ A,\Gamma\Rightarrow \neg A }$

\

$\infer[L\neg\neg] {\neg\neg A , \Gamma \Rightarrow C}{A, \neg\neg A, \Gamma \Rightarrow C}$
\quad 
$\infer[R\neg\neg] {\Gamma \Rightarrow \neg\neg A} {\Gamma \Rightarrow A}$
\end{center}
\end{defn}

The differences between \gtfm\ and \gfm\ are the following: (i) the presence of multisets of formulas in the axiom rule; (ii)  the repetition of the principal formulas of the rules in their premises (usually referred to as contraction-absorbing rules, or rules obtained by  Kleene’s trick  -- see Remark~\ref{rem.kleene.trick} below); (iii)  the absence of explicit structural rules.
(i) and (ii)  are  motivated by the need to establish the admissibility of  structural rules, which we will do below. These proofs will make precise the specific role played by each of these modifications.

\begin{rem} (Kleene's trick)   \   \label{rem.kleene.trick}
    
\m\mh 
The so called Kleene’s trick is a reformulation of the logical rules of a sequent calculus, introduced by Kleene in 
\cite[\S80]{Kleene1952}, whose main aim is to simplify decidability proofs     through the absorption of contraction  into the logical rules.

Let us compare the two systems by showing how an application of a rule in the contraction-absorbing \gtfm\ system is represented in \gfm\ where contraction is an explicit rule.
Consider the left conjunction rule of \gtfm.   

\begin{center}
$\infer[L \land]{(A \land B), \Gamma \Rightarrow C}{(A \land B),A,B, \Gamma \Rightarrow C} $
\end{center}

\mh It is easy to see that this rule is really composed of two steps: the introduction of a conjunction followed by contraction. In \gfm, an application of this rule can be represented as follows:

\begin{center}
$\infer[LC]{(A \land B), \Gamma \Rightarrow C}{\infer[L \land]{(A \land B),(A \land B), \Gamma \Rightarrow C}{(A \land B),A,B, \Gamma \Rightarrow C}} $
\end{center}


\end{rem}

\m

\subsection{Admissibility of structural rules in \gtfm}

We now show that the structural rules are admissible in \gtfm. In the following we will write $\vdash_{n}S$ (where $S$ is a sequent) to mean that $S$ has a derivation of
height at most $n$. When there is no risk of ambiguity, we omit the subscript and write
$\vdash S$ to denote the derivability of $S$.

\begin{teom}(Height-preserving weakening) \ \label{th.admis.LW}  

\m \mh 
If $\vdash_n \Gamma \Rightarrow C$ and $\Gamma \subseteq \Gamma'$, then
$\vdash_n \Gamma' \Rightarrow C$.

\begin{proof} \ 

\m\mh The proof is by induction on the height $n$ of the given derivation.
 
\m\mh If $n=0$, then $\Gamma \Rightarrow C$ is an axiom. Hence, for any $\Gamma'$ with
$\Gamma \subseteq \Gamma'$, the sequent $\Gamma' \Rightarrow C$ is also an axiom,
since adding formulas to the antecedent yields another axiom.

\m\mh For $n>0$, we consider the last inference in the derivation of
$\Gamma \Rightarrow C$.

\m\mh Suppose that $\Gamma \Rightarrow C \ :=\ \neg(A \land B), \Sigma \Rightarrow C$, i.e. 
$\Gamma \Rightarrow C$ is obtained by an application of
$L\neg\land$. Then the premises of the last inference are
$\neg A, \neg(A \land B),\Sigma \Rightarrow C$
 and $\neg B, \neg(A \land B),\Sigma \Rightarrow C$,
each one having a derivation of height strictly less than $n$. 
By the induction hypothesis, for any $\Sigma'$ such that
$\Sigma \subseteq \Sigma'$,
 we obtain derivations of
 $\neg A, \neg(A \land B),\Sigma' \Rightarrow C $   
 and 
$\neg B, \neg(A \land B),\Sigma' \Rightarrow C$
 of height strictly less than $n$.
Applying $L\neg\land$ to these sequents, we derive
$\neg(A \land B), \Sigma' \Rightarrow C$,
that is, $\Gamma' \Rightarrow C$, with a derivation of height at most $n$.

\m\mh The other cases
are left to the reader.       
\end{proof}
\end{teom}





\begin{teom} (Height-preserving contraction)   \label{th.admis.LC}

\m\mh 
If $\vdash_n A, A, \Gamma \Rightarrow C$, then $\vdash_n A, \Gamma \Rightarrow C$.
\end{teom}

\begin{proof} \  

\m\mh 
The proof is by induction on the height $n$ of the derivation.

\m\mh If $n=0$, then $A, A, \Gamma \Rightarrow C$ is an axiom, and hence
$A, \Gamma \Rightarrow C$ is also an axiom.

\m\mh For $n>0$, we consider the last inference in the derivation of
$A, A, \Gamma \Rightarrow C$.

\m\mh (i) If the contracted formula $A$ is not principal in the last inference,
then the premises of that inference have derivations of
height strictly less than $n$.
By the induction hypothesis, contraction is admissible in those premises,
and reapplying the same inference yields a derivation of
$A, \Gamma \Rightarrow C$ of height at most $n$.

\m\mh (ii) Now suppose the contracted formula  $A$ is principal in the last inference. 

\m\mh  Suppose $A = \neg(D \lor E)$. In this case,  the last rule is $L\neg\lor$ and   
the conclusion is $\neg(D \lor E), \neg(D \lor E), \Gamma \Rightarrow C$.
The premise of this inference is
$\neg D, \neg E, \neg(D \lor E), \neg(D \lor E), \Gamma \Rightarrow C$,
which has a derivation of height strictly less than $n$.
By the induction hypothesis, we obtain a derivation of
$\neg D, \neg E, \neg(D \lor E), \Gamma \Rightarrow C$
of height strictly less than $n$.
Applying $L\neg\lor$, we derive
$\neg(D \lor E), \Gamma \Rightarrow C$,
that is, $A, \Gamma \Rightarrow C$, with a derivation of height at most $n$.

\m\mh  Suppose $A=\con B$ and  the last rule is $L\con_2$.     
In this case, the conclusion is $\con B,\con B, \Gamma \Rightarrow C$, and  the premises are $\con B,\con B, \Gamma \Rightarrow B$ and
$\con B,\con B, \Gamma \Rightarrow \neg B$, both with derivations of height strictly less than $n$.
By the induction hypothesis, we obtain derivations of both premises with height strictly less than $n$.
Applying $L\con_2$, we obtain $\con B, \Gamma \Rightarrow C$,
that is, $A, \Gamma \Rightarrow C$, with a derivation of height at most $n$.

\m\mh The remaining cases are left to the reader.
\end{proof}



\mm


We now prove the admissibility of (multiplicative)  
\textit{Cut} for \gtfm\ following the format established by Dragalin  \cite{Dragalin1988}.
 The proof proceeds by a principal induction on the \emph{weight} of the cut formula, together with a secondary induction on the height of the cut, i.e.\ the sum of the heights of the premises in the corresponding application of the rule.

\begin{teom}  
(The $Cut$ rule is admissible in \gtfm)  \label{th.cut.gtfm}

\m\mh 
If $\vdash\Gamma \Rightarrow D$ and $\vdash D, \Gamma' \Rightarrow C$, then so is $\vdash \Gamma, \Gamma' \Rightarrow C$.

\begin{proof} \ 

\m\mh We need to consider three   cases: 

 1.  At least one of the premises of \textit{Cut} is an axiom;
   
   2. The cut formula is principal in both premises;
    
    3. The cut formula is not principal in at least one of the premises.

\m\mh Case 1: At least one of the premises  is an axiom. 



{\noindent 
An application of \textit{Cut} can be eliminated directly in one of two ways, depending on whether the cut formula is a principal formula or a context formula in the axiom rule.

 \mh (i) If the cut formula is principal in the axiom rule, then the two premises of the \textit{Cut} are $\Gamma,D \Rightarrow D$ and $D, \Gamma^{'} \Rightarrow C$, concluding $\Gamma, \Gamma^{'}, D \Rightarrow C$, which can be alternatively derived by zero or more applications of $LW$ to the latter premise, deleting the application of \textit{Cut} and the former premise.

 \mh (ii) If the cut formula is not principal in the axiom rule, then it is part of the context, and the two premises of the \textit{Cut} are $\Gamma \Rightarrow D$ and $D, \Gamma^{'}, A \Rightarrow A$, concluding $\Gamma, \Gamma^{'},A \Rightarrow A$, which is obviously also an axiom sequent, which can be derived without cut. 
}

\m 
\mh Case 2:  
The cut formula is principal in both premises. 
 We proceed by replacing the original $Cut$ with one of less \textit{height} and one or two of less \textit{weight}. 
 In what follows, the double line indicates one or more applications of   \textit{LC}.

\m\mh (i) The cut-formula $D = \neg (A \lor B)$.

{\begin{prooftree}
\def\fCenter{\mbox{\ $\Rightarrow$\ }}
\AxiomC{$\Gamma \Rightarrow \neg A$}
\AxiomC{$\Gamma \Rightarrow \neg B$}
\LeftLabel{$R\neg \lor$}
\BinaryInf$\Gamma \fCenter \neg (A \lor B)$
\AxiomC{$\neg A, \neg B, \neg (A \lor B), \Gamma^{'} \fCenter C$}
\RightLabel{$L\neg \lor$}
\UnaryInf$\neg (A \lor B), \Gamma^{'} \fCenter C$
\RightLabel{$Cut$}
\BinaryInf$\Gamma, \Gamma^{'} \fCenter C$
\end{prooftree} }

is replaced with 

{\footnotesize $\infer=[LC]{\Gamma,\Gamma^{'} \Rightarrow C} {\infer[Cut] {\Gamma, \Gamma, \Gamma, \Gamma^{'} \Rightarrow C} {\Gamma \Rightarrow \neg B & \infer[Cut]{\neg B, \Gamma, \Gamma, \Gamma^{'} \Rightarrow C}{\Gamma \Rightarrow \neg A & \infer[Cut]{\neg A, \neg B, \Gamma, \Gamma^{'} \Rightarrow C}{\infer[R \neg \lor]{\Gamma \Rightarrow \neg (A \lor B)}{\Gamma \Rightarrow \neg A & \Gamma \Rightarrow \neg B} & \neg A, \neg B, \neg (A \lor B), \Gamma^{'} \Rightarrow C}} }}$}


\m\mh (ii) The cut formula $D = A \lor B$. 

\begin{center}
    $\infer[Cut]{\Gamma, \Gamma^{'} \Rightarrow C}{\infer[R \lor]{\Gamma \Rightarrow A \lor B}{\Gamma \Rightarrow A} & \infer[L \lor]{A \lor B, \Gamma^{'} \Rightarrow C}{A, A \lor B, \Gamma^{'}\Rightarrow C & B, A \lor B, \Gamma^{'} \Rightarrow C}}$
\end{center}

is replaced with 

\begin{center}
    
$\infer=[LC]{\Gamma, \Gamma^{'} \Rightarrow C}{\infer[Cut]{\Gamma,\Gamma, \Gamma^{'} \Rightarrow C}{\Gamma \Rightarrow A & \infer[Cut]{A, \Gamma, \Gamma^{'} \Rightarrow C}{\infer[R \lor]{\Gamma \Rightarrow A \lor B}{\Gamma \Rightarrow A} & A, A \lor B, \Gamma^{'} \Rightarrow C}} }$

\end{center}
\m\mh The remaining cases are left to the reader. 

\m\mh 
3. The cut formula is not principal in at least one of the premises. 
In this case, we proceed by reducing the height of the cut, except when the left premise has been obtained by the rule 
$L\circ_2$.

\m\mh (i) The last rule of the right premise is $L\con_1$. 

\begin{center}
    $\infer[Cut]{\circ A,\Gamma, \Gamma^{'}\Rightarrow C}{\Gamma \Rightarrow D & \infer[L\circ_1]{\con  A, D, \Gamma^{'}\Rightarrow C}{A,\circ A,D,\Gamma^{'} \Rightarrow C & \neg A,\circ A,D,\Gamma^{'} \Rightarrow C}}$ 

\end{center}

is replaced with 

\begin{center}
 
$\infer[L\circ_1]{\circ A, \Gamma, \Gamma^{'}\Rightarrow C}{\infer[Cut]{A,\circ A, \Gamma,\Gamma^{'} \Rightarrow C}{\Gamma \Rightarrow D & A,\circ A, D, \Gamma^{'} \Rightarrow C} & \infer[Cut]{\neg A,\circ A,\Gamma, \Gamma^{'} \Rightarrow C}{\Gamma \Rightarrow D & \neg A,\circ A, D, \Gamma^{'} \Rightarrow C}}$
\end{center}

\m\mh (ii)  The last rule of the left premise is $L\neg \lor$.

\begin{center}
$\infer[Cut]{\neg(A \lor B), \Gamma, \Gamma^{'} \Rightarrow C}{\infer[L\neg \lor]{\neg(A \lor B), \Gamma \Rightarrow D}{\neg A, \neg B, \neg(A \lor B), \Gamma \Rightarrow D} & D, \Gamma^{'} \Rightarrow C}$
\end{center}

is  replaced with 

\begin{center}
$\infer[L\neg \lor]{\neg(A \lor B), \Gamma, \Gamma^{'} \Rightarrow C}{\infer[Cut]{\neg A, \neg B,\neg(A \lor B), \Gamma, \Gamma^{'} \Rightarrow C}{\neg A, \neg B, \neg(A \lor B), \Gamma \Rightarrow D & D, \Gamma^{'} \Rightarrow C}}$
\end{center}

\m\mh (iii)  The  left premise has been obtained by $L\con_2$. 

\begin{center}
$\infer[Cut]{\circ A,\Gamma,\Gamma^{'} \Rightarrow C}{\infer[L\circ_2]{\circ A, \Gamma \Rightarrow D}{\circ A, \Gamma \Rightarrow A & \circ A, \Gamma \Rightarrow \neg A} & D, \Gamma^{'}\Rightarrow C}$
\end{center}

is  replaced with 

\begin{center}
    $\infer[LW]{\circ A, \Gamma, \Gamma^{'} \Rightarrow C}{\infer[L\circ_2]{\circ A, \Gamma \Rightarrow C}{\circ A, \Gamma \Rightarrow A & \circ A, \Gamma \Rightarrow \neg A}}$
\end{center}

\m\mh The remaining cases are left to the reader.    
\end{proof}
    
\end{teom}

\begin{prop}     (Generalized subformula property)  \label{prop.subf.prop.g3}

\m\mh 
{Let $D$ be a cut-free derivation of $\Gamma \Rightarrow C$ in \gtfm. Then, all formulas appearing in $D$ will be generalized subformulas of $\Gamma $ and  $C$}
\begin{proof}
The  claim is easily established by a simple inspection of the rules given of \gtfm.     \end{proof}
\end{prop}

 \begin{rem} \ 
 
\m\mh Schellinx \cite{Schellinx1994TheNA} explored (mainly in the context of linear logic) how the use of structural rules affects the reduction steps in a cut-elimination proof, as well as how the reduction steps can give information about the derivation itself. 
When contraction or weakening are used to derive a cut-formula, a corresponding operation at the level of (sub)derivations occurs during the reduction steps: either a (sub)derivation is deleted or it is duplicated. 
This allows for an interesting observation of how structural rules, even when ``absorbed'' by logical rules and made implicit, still produce similar effects of duplication, in the case of contraction, and deletion, in the case of weakening. 
In the proof above, this behavior is  visible in Case~2, since
contraction is ``built into'' the logical rules (see
Remark~\ref{rem.kleene.trick}),
 and in sub-case (iii) of Case 3, where weakening is implicit in the $L\circ_2$ rule.

\end{rem}

\begin{rem} (On the path to \gtfm) \ \label{rem.path.gtfm}
        

In order to obtain a fully logical sequent calculus for \letfm, we must establish the admissibility of the structural rules.
Weakening is straightforward: its admissibility could be achieved for our systems just by adding an
axiom rule with contexts. 
Contraction, however, is not. 

Our first attempt consisted in dropping contraction, cut, and weakening from \gfm, adding the new axiom rule with contexts, and then trying to prove the height-preserving admissibility of contraction.
This was a natural move, since the logical rules of \gfm\ already resemble those of standard \textit{G3} systems for intuitionistic logic, where the result follows by a simple induction on derivation height.
Unfortunately, this did not work, since the proof requires a height-preserving inversion lemma for the left rules.
Intuitively, the demonstration proceeds by showing that contraction can be ``pushed up'' the derivation by applying contraction to the subformulas of   the contracted formula until it is eliminated. This can be achieved with the help of the height-preserving inversion lemma, which cannot be established for the \con-rules. Therefore, in order to make progress, we had to modify our calculus.



This led us to an intermediate system, all its logical rules coincided with those of our slightly modified and structural rule-free \gfm, except for the \con-rules, which we reformulated as in \gtfm \footnote{\textit{G3} calculi for intuitionistic logic, like \textit{G3ip} \cite{Negri2001}, usually adopt a similar strategy, using Kleene's trick on the left premise of the left implication rule to deal with its non-invertibility in contraction admissibility proofs.}.
Even so, the proof did not go through: it failed because we cannot restrict the axiom rule to literal principal formulas (as explained in Section~\ref{sec.gfm}), which blocks the proof of the height-preserving invertibility lemma, even for the non-\con\ rules.


It was this failed attempt with the intermediate system that led us to apply Kleene’s trick to all logical rules of \gfm\ to obtain \gtfm. This, in turn, enabled us to prove the admissibility of height-preserving $LC$ without appealing to any invertibility lemma.

    \end{rem}

\subsection{Soundness and completeness} \label{sec.sound.comp}

We first establish the equivalence between \gfm\ and \gtfm. 


\begin{teom}  \ \label{th.equiv.G1G3}

\m\mh 
{$\gfm\vdash \Gamma \Rightarrow C$    if and only if $\gtfm\vdash \Gamma \Rightarrow C$}

\begin{proof}  
From left to right, all left structural rules, together with cut, were shown to be admissible in \gtfm\ in theorems  \ref{th.admis.LW}, \ref{th.admis.LC}, and \ref{th.cut.gtfm}.  All left logical rules of \gfm\ can also be shown admissible in \gtfm\ through the help of Theorem~\ref{th.admis.LW}.
 From right to left, all rules of \gtfm\ can also be shown admissible in \gfm. The axiom rule follows from 0 or more uses of $LW$, and the other logical rules follow from a simple application of $LC$.
\end{proof}
\end{teom}

Now, soundness and completeness of \gtfm\ follow straightforwardly from the above result. 

\begin{teom}  \

\m
\mh \gtfm\ is sound and complete with respect to the bivalued semantics of \letfm:  $\Gamma\vdash_{\gtfm} A$ \ if and only if \ $\Gamma\vDash_\letfm  A$ 

\begin{proof}
 The result follows from Theorem~\ref{th.equiv.G1G3} above, together with  Theorems~\ref{th.sound.comp.nfm} and \ref{th.equiv.gn}. 
\end{proof}
\end{teom}

\subsection{A decision method for \letfm} \label{sec.decid}




Recall that the design of our rules is inspired by Kleene’s \textit{G3} systems. He introduced these systems primarily to obtain decidability results via a bottom-up proof-search procedure, precisely to control difficulties arising from contraction. In what follows, we follow the same lines as Kleene’s approach in \cite[\S80]{Kleene1952} to prove the decidability of \gtfm.

\begin{defn} \  (Cognate sequents)

\m\mh (i) Two sequents $\Gamma \Rightarrow A$ and $\Gamma' \Rightarrow A'$ are said to be \textit{cognate} if $\Gamma$ and $\Gamma'$ contain the same formulas, possibly differing in multiplicity, that is, in the number of occurrences of each formula, and $A = A'$.

\m\mh 
(ii) An \textit{irredundant} proof in \gfm\ is a proof in which no pair of \textit{cognate} sequents occurs one above the other along the same derivation branch.

\end{defn}

\begin{prop}   (Irredundant derivations) \label{prop.irred.deriv} 

\m\mh  
    Every cut-free derivation $\mathcal D$ of $\Gamma \Rightarrow C$ in \gtfm\ can be transformed into an \emph{irredundant} 
    cut-free derivation $\mathcal D^{*}$ of the same sequent.

\begin{proof} \ 

\m\mh If no branch of $\mathcal D$ contains a pair of cognate sequents such that one
occurs above the other (possibly with intervening sequents), then $\mathcal D$
is already irredundant.

\m\mh Otherwise, fix a branch and choose a pair of cognate sequents $S_1$ and $S_2$ on that branch such that $S_1$ occurs above $S_2$ and there is no cognate sequent strictly between them on that branch. Write $S_1$ as $\Sigma \Rightarrow A$ and $S_2$ as $\Sigma' \Rightarrow A$, where $\Sigma$ and $\Sigma'$ contain the same formulas, possibly with different multiplicities.

\m\mh Delete the subderivation leading from $S_1$ to $S_2$ ($S_2$ included) together with all branches contributing to it. 
Next,    apply   Theorem~\ref{th.admis.LC}  to $S_1$ (possibly zero times), rendering it equal to $S_2$.
The resulting derivation is still cut-free, derives $\Gamma \Rightarrow C$,
and strictly decreases the number of occurrences, along the branches of
$\mathcal D$, of pairs of cognate sequents with one occurring above the other.

\m\mh By iterating this transformation, and since the number of occurrences of cognate sequents is finite, the process must terminate, yielding a cut-free, irredundant derivation of $\Gamma \Rightarrow C$.
 \end{proof}
\end{prop}

The proof that \gtfm\ provides a decision method to \letfm\   will be given after the following 
example illustrating the decision procedure.

\begin{exem} \ \label{ex.silog.disj} 

\mm\mh 
(i) The sequent $\neg p, p\lor q \Rightarrow q$ (disjunctive syllogism) is not derivable in \gtfm.   

\m 
\begin{center}   
$\infer[L\lor]{\neg p, p\lor q \Rightarrow q}{\infer[L\lor]{\neg p, p, p\lor q\Rightarrow q}{ \neg p, p, p, p\lor q\Rightarrow q & \neg p, p, q, p\lor q \Rightarrow q } & \neg p, q, p\lor q \Rightarrow q}$    
    
\m \justifying{\mh Observe that further bottom-up applications of $L\lor$ will keep producing axioms on the right premise and cognate sequents on the left premise. The example above   shows that disjunctive syllogism   is not derivable, as expected, since in the context of \lets\ 
it is equivalent to explosion. Below,   we   show how its derivability can be recovered by means of the classicality connective. }    
    
    \end{center}
\mh (ii) The sequent $\con p, \neg p, p\lor q \Rightarrow q$   is  derivable in \gtfm.   

 \begin{center}
$\infer[L\lor]{\con p, \neg p, p\lor q \Rightarrow q}{\infer[L\con_2]{\con p, 
\neg p, p, p\lor q \Rightarrow q}{ \con p, \neg p, p, p\lor q \Rightarrow p & \con p, \neg p, p, p\lor q \Rightarrow \neg p} & \con p, \neg p, q, p\lor q \Rightarrow q}$

\m \justifying{\mh A bottom-up application of the rule $L\lor$ yields an axiom as the right premise and a sequent as the left 
premise that is clearly derivable, since it is an instance of $\Gamma,\con A, A, \neg A \vdash B$. 
The latter is derivable in \gtfm\ and  recovers the principle of explosion. By a further bottom-up application of $L\con_2$, 
we obtain two axioms and close the proof.}    

\end{center}



\mh (iii)
 The sequent $\Rightarrow \con p \lor\neg \con p$  is not derivable in \gtfm.   


\begin{center}

$\infer[R\lor]{\Rightarrow \con p \lor \neg \con p}{\Rightarrow \con p}$
\quad \quad  
$\infer[R\lor]{\Rightarrow \con p \lor \neg \con p}{\Rightarrow \neg \con p}$

\m\justifying{\mh It is clear from the constructions above that all possibilities for finding a proof have  been considered. 
No further inference can be made from the premises resulting from bottom-up applications of $R\lor$.}
\end{center}

\mh (iv)
 The sequent $\con   p  \Rightarrow   p$   is  not derivable in \gtfm.   

\m\mh {Consider the constructions below:}

\begin{center}
{ \small $\infer[L\con_2]{\con p \Rightarrow p}{\con p \Rightarrow p & \con p  \Rightarrow \neg p}$
\quad  \quad
 $\infer[L\con_1]{\con p \Rightarrow p}{\con p, p \Rightarrow p & \infer[L\con_1]{\con p, \neg p \Rightarrow p}{\con p, \neg  p, p \Rightarrow p &   \con p, \neg p, \neg p \Rightarrow    p}}$}

   \m  \justifying{\mh     
The construction on the left shows that a bottom-up application of $L\con_{2}$ yields a cognate sequent as the left premise, thereby stopping the search in that branch.
Note that any application of $L\con_{2}$ to a sequent of the form $\Gamma, \con A \Rightarrow A$ or $\Gamma, \con A \Rightarrow \neg A$ yields a cognate sequent as the left or right premise.  
Thus, the only remaining possibility is to apply $L\con_{1}$.
However, two successive bottom-up applications of $L\con_{1}$ also yield
cognate sequents, as the construction on the right shows. }
\end{center}

\vv\vv 




\mh We leave as an exercise to the reader to show that the sequents     
$p\land\neg p\Rightarrow q$ and  $\Rightarrow \con p \land  p \land\neg p$ are not derivable in \gtfm. 
 The non-derivability of the latter sequent shows that \gtfm\ is syntactically
consistent. Indeed, in \letfm\  a bottom particle can be defined by
$\bot \defi \con A \land A \land \neg A$, for any $A$. Hence, if $\Rightarrow \bot$ were
derivable, then so would be $\Rightarrow \con p \land p \land \neg p$, which is
impossible. Therefore, $\Rightarrow \bot$ is not derivable.

\end{exem}

The reader should have noticed how the proof method proceeds.
If the sequent is derivable, a choice of bottom-up rule applications that avoids producing cognate sequents will 
invariably end with axioms on every branch. 
If the sequent is not derivable, 
the proof search will eventually produce a cognate sequent, at which point the search stops. 
All possible bottom‑up rule applications must be examined to confirm that no derivation can be completed.

\begin{teom} \  \label{th.decid.gtfm}

\m\mh 
    The question of determining whether or not a sequent $\Gamma \Rightarrow C$ is derivable in \gtfm\ is decidable.

    \begin{proof}
 
      Our strategy will be to show that we can generate all  possible cut-free irredundant derivations of $\Gamma \Rightarrow C$ by constructing those derivations root first, choosing step by step a principal formula of the bottom-most sequent to analyze into its possible premises.

\m\mh 
By Proposition~\ref{prop.subf.prop.g3} (generalized subformula property) and Theorem~\ref{th.cut.gtfm} (admissibility of cut), any derivation of \(\Gamma \Rightarrow C\) can be assumed to be cut‑free and to contain only generalized subformulas of \(\Gamma\) and \(C\).

\m\mh Proposition~\ref{prop.irred.deriv} allows us to restrict the search to irredundant derivations; that is, we may stop extending a branch whenever a pair of subsequent cognate sequent appears.

\m\mh Since the set of generalized subformulas of \(\Gamma\) and \(C\)  
 is finite and  we may restrict to irredundant derivations (which avoid cognate sequents), only finitely many distinct sequents can appear in such a derivation. 
 
\m\mh  Moreover, each inference rule has a fixed number of premises, so there are only finitely many ways to apply a rule to a given sequent.

\m\mh Consequently, the bottom-up  proof search tree is finite: every branch terminates either with an axiom or with a sequent to which no further rule can be applied without introducing a cognate sequent.

\m\mh Thus, the number of possible irredundant, cut‑free derivations built from generalized subformulas is finite. If among these we find a derivation whose leaves are all axioms, then $\Gamma \Rightarrow C$  is derivable; otherwise, it is not.
    \end{proof} 
    \end{teom}

\section{What's next}   \label{sec.final.rem}

As future work, we plan to extend this proof-theoretic investigation to
other logics in the \lets\ family, by developing \textit{G3}-style sequent calculi
together with bottom-up decision procedures for other sentential systems, in
particular \letj\ and the six-valued logics \letfp\ and \letkp\ \cite{Coniglio.Rod.SL.2023}. 
In a second stage, we also intend to investigate the proof theory of their corresponding first-order extensions.

\bibliographystyle{plainnat}  
 \bibliography{library}

 \end{document}